\documentclass[11pt,centertags,reqno,twoside]{amsart}
\usepackage{amssymb}
\usepackage{mathptmx}
\usepackage{graphicx}

\usepackage[english]{babel}

\usepackage{color}
\usepackage[
breaklinks,colorlinks,
naturalnames,
unicode,
linkcolor=blue,urlcolor=blue,filecolor=blue,citecolor=blue,
anchorcolor=blue,menucolor=blue,pagecolor=blue,
pdfpagemode=UseOutlines,
pdfcenterwindow=true,
pdfstartview=FitH,
pdfnewwindow=true,
pdftitle={Comment on The tan theta theorem with relaxed conditions by Y. Nakatsukasa},
pdfauthor={Alexander K. Motovilov},
pdfkeywords={Davis-Kahan, tan theta theorem, a posteriori tan theta theorem, a priori tan theta theorem, largest canonical angle, largest principal angle},
pdfsubject={commentary paper}]{hyperref}
\definecolor{blue}{rgb}{0.1,0.15,0.5}
\newcommand{\hnl}{\htmladdnormallink}

\DeclareSymbolFont{SY}{U}{psy}{m}{n}
\DeclareMathSymbol{\emptyset}{\mathord}{SY}{'306}

\renewcommand{\eqref}[1]{{\rm(\ref{#1})}}

\newcommand{\bbC}{{\mathbb C}}

\newcommand{\fA}{\mathfrak{A}}

\newcommand{\fL}{\mathfrak{L}}

\newcommand{\fQ}{\mathfrak{Q}}

\newcommand{\fX}{\mathfrak{X}}

\newcommand{\diag}{\mathop{\rm diag}}

\newcommand{\be}{\begin{equation}}
\newcommand{\ee}{\end{equation}}

\allowdisplaybreaks

\newtheorem{theorem}{Theorem}

\newtheorem{lemma}[theorem]{Lemma}
\newtheorem{proposition}[theorem]{Proposition}

\theoremstyle{definition}

\theoremstyle{remark}
\newtheorem{remark}[theorem]{Remark}

\begin{document}

\title[C\lowercase{omment on} `T\lowercase{he $\tan\theta$ theorem with relaxed conditions}']
{\large C\lowercase{omment on} `T\lowercase{he tan}$\,\theta$
\lowercase{theorem with relaxed conditions', by}
Y.\,N\lowercase{akatsukasa}}

\author[A. K. Motovilov]
{Alexander K. Motovilov}

\address{Alexander K. Motovilov, Bogoliubov Laboratory of
Theoretical Physics, JINR, Joliot-Cu\-rie 6, 141980 Dubna, Moscow
Region, Russia} \email{motovilv@theor.jinr.ru}




\begin{abstract}
We show that in case of the spectral norm, one of the main results
of the recent paper \hnl{\textit{The tan $\theta$ theorem with
relaxed conditions}}{http://dx.doi.org/10.1016/j.laa.2011.08.038},
by Yuji Nakatsukasa, published in \textit{Linear Algebra and its
Applications} is a corollary of the $\tan\theta$ theorem proven in
[V. Kostrykin, K.\,A. Makarov, and A.\,K. Motovilov, \textit{On the
existence of solutions to the operator Riccati equation and the
tan\,$\theta$ theorem},  IEOT \hnl{\textbf{51}
(2005), 121 -- 140}{http://dx.doi.org/10.1007/s00020-003-1248-6}].
We also give an alternative finite-dimensional matrix formulation of
another $\tan\theta$ theorem proven in [S.\,Albeverio and
A.\,K.\,Motovilov, \textit{The a priori tan\,$\theta$ theorem for spectral
subspaces}, IEOT \hnl{\textbf{73} (2012),
413 –- 430}{http://dx.doi.org/10.1007/s00020-012-1976-6}].
\end{abstract}

\maketitle

In a recent paper \cite{Naka} published in \textit{Linear Algebra
and its Applications}, Y. Nakatsukasa obtains two bounds on the
tangent of the canonical angles between an approximate and an exact
spectral subspace of a Hermitian matrix. These bounds (see
\cite[Theorems 1 and 2]{Naka}) extend respectively the $\tan\theta$
theorem and the generalized $\tan\theta$ theorem proven by
C.\,Davis and W.\,M.\,Kahan in their celebrated paper~\cite{DK70}.
Actually, an extension of the $\tan\theta$ theorem similar to
\cite[Theorem~1]{Naka} has already been given in \cite{KMM3}, in the
wider context of the perturbation theory for self-adjoint operators
on a Hilbert space.

In our discussion below we restrict ourselves to the spectral norms
of the matrices involved, that is, by $\|S\|$ we always understand
the maximal singular value of a matrix $S$. If $\fA$ and $\fL$ are
subspaces of $\bbC^n$, the notation $\angle(\fA,\fL)$ is used for
the largest principal angle between $\fA$ and~$\fL$.

We begin with presenting a relevant finite-dimensional version of
the $\tan\theta$ theorem from \cite{KMM3} (see
\cite[The\-o\-rem~2]{KMM3}).

\begin{proposition}
\label{Prop1}
Assume that a Hermitian matrix $L\in\bbC^{n\times n}$ is
block partitioned in the form
\begin{equation}
\label{Lt}
L=\left[\begin{array}{cc}
A_1 & B^H \\
B & A_2
\end{array}
\right]
\end{equation}
with $A_1\in\bbC^{k\times k}$, $1<k<n$. Let the spectrum of $A_1$
lie in $(-\infty,\alpha-\delta]\cup[\beta+\delta,\infty)$, where
$\alpha\leq\beta$ and $\delta>0$. Suppose that $\fL_1$ and $\fL_2$ are
complementary orthogonal reducing subspaces of $L$ such that
$\dim(\fL_1)=k$ and the spectrum of the restriction
$L\bigr|_{\fL_2}$ of (the operator) $L$ on the reducing subspace
$\fL_2$ is confined in $[\alpha,\beta]$. Also, let $\fA_1$ be the
subspace of $\bbC^n$ spanned by the first $k$
columns of the identity matrix $I_n$. Then
\begin{equation}\label{TanKMM}
\tan\angle(\fA_1,\fL_1)\leq\frac{\|B\|}{\delta}.
\end{equation}
\end{proposition}

\begin{remark}
Actually, \cite[Theorem 2]{KMM3} (combined with \cite[Theorem
2.3]{KMM3}) suggests the equivalent bound
\,$\delta\tan\|{\Theta}\|\leq {\|B\|}$\, for the operator angle
${\Theta}$ between the orthogonal complements $\fA_2$ and $\fL_2$ of
the subspaces $\fA_1$ and $\fL_1$, respectively, provided that
$\fL_2$ is the graph of an operator from $\fA_2$ to $\fA_1$. But the
latter, in the finite-dimensional case under consideration, holds
true automatically. This is seen from the following lemma.
\end{remark}

\begin{lemma}
Assume the hypothesis of Proposition \ref{Prop1}. Then
$\fA_2\cap\fL_1=\fA_1\cap\fL_2=\{0\}$ and, hence, the reducing
subspace $\fL_2$ is the graph of an operator from $\fA_2$ to $\fA_1$.
\end{lemma}

\begin{proof}
By the hypothesis, the dimensions of the subspaces $\fL_1$ and $\fA_1$
coincide, $\dim (\fL_1)=\dim (\fA_1)=k$. Then by using the canonical
orthogonal decomposition of $\bbC^n$ with respect to the orthogonal
projections onto $\fA_1$ and $\fL_1$ (see, e.g. \cite[Theorem
2.2]{KMM2}) one verifies that
\begin{equation}
\label{dims}
\dim(\fA_2\cap\fL_1)=\dim(\fA_1\cap\fL_2).
\end{equation}
Suppose that
$\fA_1\cap\fL_2\neq\{0\}$. In such a case, there is a vector $y\in\fL_2$ of the
form $y=\left[\hspace*{-0.5ex}\begin{array}{c} x \\
0_{n-k}\end{array}\hspace*{-0.5ex}\right]$, where the lower subcolumn
$0_{n-k}$ consists of exactly $n-k$ zeros and the upper subcolumn $x$ contains
at least one nonzero element. For $c=(\alpha+\beta)/2$ one then obtains
$$
\|(L-cI_n)y\|^2=\|(A_1-cI_k)x\|^2+\|Bx\|^2\geq\|(A_1-cI_k)x\|^2\geq
\bigl(\mbox{$\frac{1}{2}$}(\beta-\alpha)+\delta\bigr)^2\|y\|^2,
$$
since $\|y\|=\|x\|$ and the spectrum of $A_1$ belongs to
$(-\infty,\alpha-\delta]\cup[\beta+\delta,\infty)$. On the other hand, for
$y\in\fL_2$ we should have
$\|(L-cI_n)y\|\leq \mbox{$\frac{1}{2}$}(\beta-\alpha)\|y\|$ since the spectrum
of the restriction $L\bigr|_{\fL_2}$ lies in $[\alpha,\beta]$. Hence, $y=0$, a
contradiction, which yields $\fA_1\cap\fL_2=\{0\}$. Taking into
account \eqref{dims} one concludes that also $\fA_2\cap\fL_1=\{0\}$.
Applying \cite[The\-o\-rem 3.2]{KMM2} completes the proof.
\end{proof}

Now we show that for the spectral norm the $\tan\theta$
theorem proven in \cite{Naka} is a corollary of Pro\-po\-si\-ti\-on
\ref{Prop1}. We reproduce the corresponding statement
from \cite{Naka} in the following form  (see \cite[Theorem~1]{Naka}).

\begin{proposition}[\cite{Naka}]
\label{Prop2}
Let $A\in\bbC^{n\times n}$ be
a Hermitian matrix. Let $X=[X_1\,\,X_2]$ be a unitary eigenvector matrix
of $A$ with $X_1\in\bbC^{n\times k}$, $1<k<n$, so that
$X^HAX=\diag(\Lambda_1,\Lambda_2)$ is diagonal and $\Lambda_1$ has
$k$ columns. Assume that the columns of a matrix $Q_1\in\bbC^{n\times k}$ are
orthonormal and let $R=AQ_1-Q_1A_1$, where $A_1=Q_1^HAQ_1$.
Furthermore, assume that for some $\alpha\leq\beta$ and $\delta>0$ the spectrum of $A_1$
lies in $(-\infty,\alpha-\delta]\cup[\beta+\delta,\infty)$ and the spectrum of $\Lambda_2$
belongs to $[\alpha,\beta]$. Then
\begin{equation}
\label{tanN}
\tan\angle(\fQ_1,\fX_1)\leq\frac{\|R\|}{\delta},
\end{equation}
where $\fQ_1$ and $\fX_1$ are the subspaces spanned by the columns
of $Q_1$ and $X_1$, respectively.
\end{proposition}

\begin{proof}
Assume that $Q_1$ is a submatrix of a unitary $n\times n$ matrix
$Q=[Q_1\,\,Q_2]$ and let $L=Q^HAQ$. The matrix $L$ has the form
\eqref{Lt} with $A_1=Q_1^HAQ_1$, $A_2=Q_2^HAQ_2$, and $B=Q_2^HAQ_1$.
Since $A$ is unitarily equivalent to the diagonal matrix
$\Lambda=\diag(\Lambda_1,\Lambda_2)$, the same is true for $L$.
Moreover, the $k$-dimensional subspace $\fL_1=Q^H\fX_1$ and its
orthogonal complement $\fL_2=\bbC^n\ominus\fL_1$ are reducing
subspaces of $L$. The spectrum of the restriction $L\bigr|_{\fL_2}$
coincides with the spectrum of $\Lambda_2$ and, hence, it lies in
$[\alpha,\beta]$. If the subspace $\fA_1$ is as in Proposition
\ref{Prop1}, then, just  by this
proposition, the largest principal angle between $\fA_1$ and $\fL_1$
satisfies the bound \eqref{TanKMM}. Meanwhile, the subspaces
$\fQ_1$ and $\fX_1$ are obtained from $\fA_1$ and $\fL_1$ by the
same unitary transformation: $\fQ_1=Q\fA_1$ and $\fX_1=Q\fL_1$.
Hence, $\angle(\fQ_1,\fX_1)=\angle(\fA_1,\fL_1)$. Observing that
$B=Q^H_2(AQ_1-Q_1A_1)=Q^H_2R$,\, one infers $\|B\|=\|R\|$ and then
\eqref{TanKMM} implies \eqref{tanN}.
\end{proof}

\begin{remark}
In its turn, Proposition  \ref{Prop1} may be viewed as a particular
version of Proposition \ref{Prop2} for the case where $[Q_1\,\,Q_2]$ is
taken equal to the identity matrix $I_n$. Thus,
in fact these two propositions are equivalent to each other.
\end{remark}

We next note that there is another sharp $\tan\theta$ bound
established in \cite[Theorem 1]{AM2010} (see also \cite[Theorem 2]{MotSel} for
an earlier result). The following
assertion represents a finite-dimensional version of
\cite[Theorem~1]{AM2010} reformulated in the style of Proposition
\ref{Prop2}.

\begin{proposition}
\label{Prop3} Let $A\in\bbC^{n\times n}$ be a Hermitian matrix and
$Q=[Q_1\,\,Q_2]$ a unitary matrix with $Q_1\in\bbC^{n\times k}$,
$1<k<n$. Assume that for some $a\leq b$ and $d>0$ the spectrum of
$A_1=Q_1^HAQ_1$ lies in $(-\infty,a-d]\cup[b+d,\infty)$ and that the
spectrum of $A_2=Q_2^HAQ_2$ belongs to $[a,b]$. Let $R=AQ_1-Q_1A_1$
and suppose that $\|R\|<\sqrt{2}d$. Then $n$ orthonormal
eigenvectors of $A$ may be numbered in such an order that the
corresponding unitary eigenvector matrix $X=[X_1\,\,X_2]$ with
$X_1\in\bbC^{n\times k}$ reduces $A$ to the diagonal form
$X^HAX=\diag(\Lambda_1,\Lambda_2)$ with $\Lambda_1\in\bbC^{k\times
k}$ having its spectrum in $(-\infty,a-d]\cup[b+d,\infty)$, and with
$\Lambda_2$ having all its eigenvalues in $[a-\delta_R,b+\delta_R]$,
where
$\delta_R=\|R\|\tan\left(\frac{1}{2}\arctan\frac{2\|R\|}{d}\right)<d$.
Moreover,
\begin{equation}
\label{tanA}
\tan\angle(\fQ_1,\fX_1)\leq\frac{\|R\|}{d},
\end{equation}
where $\fQ_1$ and $\fX_1$ are the subspaces spanned by the columns
of $Q_1$ and $X_1$, respectively.
\end{proposition}
\begin{proof}
The matrix $L=Q^HAQ$ has the form \eqref{Lt} with $A_1$ and $A_2$
defined in the hypothesis, and $B=Q^H_2AQ_1$. As in the proof of
Proposition \ref{Prop2} we have $\|B\|=\|R\|$. Hence
$\|B\|<\sqrt{2}d$ and then the statement on the eigenvalue matrix
$\Lambda$ and, in particular, on the spectral inclusions for
$\Lambda_1$ and $\Lambda_2$, is an immediate corollary of
\cite[Theorem 2]{KMM4}. Furthermore, for the case under
consideration, the bound from \cite[estimate (1.3) in Theorem
1]{AM2010} may be equivalently written as
\,$d\tan\angle(\fA_1,\fL_1)\leq{\|R\|}$, where $\fA_1$ is as in
Proposition \ref{Prop1} and $\fL_1$ is the spectral subspace of $L$
associated with the set $(-\infty,a-d]\cup[b+d,\infty)$. By the
unitarity argument we already used in the proof of Proposition
\ref{Prop2}, the bound \,$d\tan\angle(\fA_1,\fL_1)\leq{\|R\|}$\, implies
the bound~\eqref{tanA}.
\end{proof}

\begin{remark}
In general, condition $\|R\|<\sqrt{2}d$ cannot be removed. If this
condition is violated, the matrix $A$ may not have eigenvalues in
the interval $(a-d,b+d)$ at all (see \cite[Example 1.6]{KMM4}).
\end{remark}

If we estimate $\angle(\fQ_1,\fX_1)$ by using inequality
\eqref{tanA}, no knowledge on the exact eigenvalues of $A$ is
required. Unlike the bound \eqref{tanN}, the estimate \eqref{tanA}
involves the separation distance $d$ between the respective
eigenvalue sets of the matrices $A_1$ and $A_2$.  In applications,
these sets are usually treated as an approximate spectrum of $A$ and
their separation distance is assumed to be known prior to further
calculations. Following \cite{MotSel} and \cite{AM2010}, it is
appropriate thus to call the bound \eqref{tanA} the \textit{a priori}
$\tan\theta$ theorem. Similarly, the bound \eqref{tanN} may be
called the \textit{(semi-)a posteriori} $\tan\theta$ theorem since it
involves the separation distance $\delta$ between one approximate
and one exact spectral sets.

\vspace*{2mm} \noindent {\small{\bf Acknowledgments.} The author gratefully
acknowledges financial support of his work by the Deutsche
For\-sch\-ungs\-gemeinschaft (DFG) and by the Russian Foundation for
Basic Research.}


\end{document}